\documentclass[12pt]{amsart}

\usepackage{amsmath,amstext,amssymb,amsopn,amsthm}
\usepackage{url,verbatim}
\usepackage{mathtools}
\usepackage{enumerate}

\usepackage{color,graphicx}

\usepackage[margin=30mm]{geometry}
\usepackage{eucal,mathrsfs,dsfont}%gives nicer set-names. 

\numberwithin{equation}{section}

\allowdisplaybreaks

\usepackage[utf8]{inputenc}
\usepackage{graphicx,amsmath,amssymb,enumerate,color,amsthm}
\usepackage{hyperref}
\usepackage{epstopdf}
\renewcommand{\P}{\mathbb{P}}

\newcommand{\ED}{\mathrm{ED}}

\newcommand{\R}{\mathbb{R}}
\newcommand{\N}{\mathbb{N}}

\newcommand{\E}{\mathbb{E}}
\newcommand{\1}{\mathbf{1}}

\newcommand{\dd}{\mathrm{d}}

\newtheorem{thm}{Theorem}[section]
\newtheorem{lemma}[thm]{Lemma}
\newtheorem{corol}[thm]{Corollary}

\theoremstyle{definition}

\newtheorem{defin}[thm]{Definition}

\newtheorem{remark}[thm]{Remark}

\newcommand{\LDM}{\mathrm{LDM}}

\def\ess{{\mathrm{ess}}}

\title[Perpetuities with light tails and the local dependence measure]{Perpetuities with light tails \\and the local dependence measure}

\author{Julia Le Bihan}
\email{julia.le\_bihan.stud@pw.edu.pl}

\author[B. Ko\l{}odziejek]{Bartosz Ko\l{}odziejek}
\email{bartosz.kolodziejek@pw.edu.pl}
\address{Faculty of Mathematics and Information Sciences, Warsaw University of Technology, Koszykowa 75, \mbox{00-662} Warsaw, Poland}

\thanks{ 
BK: This research was funded in part by National Science Centre, Poland, 2023/51/B/ST1/01535.
\\ For the purpose of Open Access, the authors have applied a CC-BY public copyright licence to any Author Accepted Manuscript (AAM) version arising from this submission.}

\keywords{Tail estimates, stochastic fixed point equation, iterated random sequence, perpetuity}
\subjclass{Primary 60H25, 60G10, 37M10; Secondary 60J05}
\begin{document}

\begin{abstract}
This work investigates the tail behavior of solutions to the affine stochastic fixed-point equation of the form $X\stackrel{d}{=}AX+B$, where $X$ and $(A,B)$ are independent. Focusing on the light-tail regime,
following [Burdzy et al. (2022), Ann. Appl. Probab.] we introduce a local dependence measure along with an associated Legendre-type transform. These tools allow us to effectively describe the logarithmic right-tail asymptotics of the solution $X$. 

Moreover, we extend our analysis to a related recursive sequence $X_n=A_n X_{n-1}+B_n$, where $(A_n,B_n)_{n}$ are i.i.d. copies of $(A,B)$. For this sequence, we construct deterministic scaling $(f_n)_{n}$ such that $\limsup_{n\to\infty} X_n/ f_n$ is a.s. positive and finite, with its non-random explicit value provided.
\end{abstract}

\maketitle

\section{Introduction}
We study the tail behavior of a solution $X$ to the stochastic fixed-point equation 
\begin{align}\label{eq:affine}
X\stackrel{d}{=}AX+B,\qquad X\mbox{ and }(A,B)\mbox{ are independent}.
\end{align}
 When such a solution exists, it is referred to as a perpetuity. to Our primary focus is on the light-tail regime. Throughout the paper, we assume that  
$A$ and $B$ are almost surely nonnegative and nonzero with positive probability, thus excluding trivial cases. For comprehensive background material, we refer to \cite{BurBook16}.

The influence of the joint distribution of $(A,B)$ on the tail behavior of $X$ is notably diverse. When conditions ensure a light-tailed solution (see \cite{GG96, HitWes09, Hit10, BK17, BKT22}), the analysis of \eqref{eq:affine} differs fundamentally from that in the heavy-tailed setting (cf. \cite{Kes73, Kev1, Kev2, Gol91, Gri75, DK18}). Roughly, the distinction between the light- and heavy-tailed cases hinges upon whether $\P(A\leq 1)=1$ with light-tailed $B$ or $\P(A>1)>0$. 
In particular, for light-tailed distributions, the dependence structure between $A$ and $B$ significantly affects the tail asymptotics. Conversely, for heavy-tailed solutions, marginal distributions predominantly determine asymptotics, with dependence structure influencing only multiplicative constants (see e.g. \cite{Dysz16, Kev1, DD18}).

Our methodology builds upon the framework established by \cite{BKT22}, which analyzed the left-tail behavior (near $0^+$) of $X$. Following these ideas, we introduce a local dependence measure ($\LDM$) along with its associated Legendre-type transform. These tools form the cornerstone of our analysis and effectively describe the logarithmic right-tail asymptotics of $X$. 

Specifically, we assume the following limit exists for all $y > 0$: (a $\LDM$)
\[
g(y)=\lim_{t\to\infty} \frac{\log\P(A ty + B > t)}{\log\P(B>t)},
\]
where the function $t\mapsto -\log\mathbb{P}(B>t)$ is regularly varying with index $\rho>0$. 
Our first main result (Theorem \ref{theorem4.1}) shows that if a solution to \eqref{eq:affine} exists and the function $g$ satisfies a technical condition (termed admissibility), then 
\[
\lim_{t\to\infty} \frac{\log\P(X>t)}{\log\P(B>t)} = \lambda^\ast,
\]
where $\lambda^\ast$ is the unique nonzero fixed point of the function $\phi_\rho(\lambda) = \inf_{y>0}\{y^\rho \lambda + g(y)\}$. Moreover, we show that $\lambda^\ast=\inf_{y\in(0,1)}\left\{\frac{g(y)}{1-y^\rho}\right\}$.

It is natural, given \eqref{eq:affine}, to consider the recursive sequence
\begin{align}\label{eq:seq}
X_n = A_n X_{n-1} + B_n, \quad n \geq 1,
\end{align}
where $(A_n, B_n)_{n \geq 1}$ are i.i.d. copies of $(A,B)$, and $X_0$ is independent of this sequence. This sequence forms a Markov chain whose stationary distribution solves \eqref{eq:affine}. Our second main result (Theorem \ref{nowe6.1}) identifies an upper envelope for $(X_n)_{n\geq 0}$ when $X_0 = 0$. We construct a deterministic scaling sequence $(f_n)_n$ such that
\[
\limsup_{n\to\infty}\frac{X_n}{f_n}
\]
is almost surely positive and finite, with its deterministic explicit value provided. To the best of our knowledge, this constitutes the first characterization of the upper envelope for a sequence defined by \eqref{eq:seq}. The proof extends techniques from \cite{BKT22}, which derived a lower envelope under related but distinct assumptions.

Our study is inspired by several previous investigations into the tail behavior of $X$. A common strategy in earlier work (cf. \cite{Gri75, Gre94, Palmowski07, BDIM}) involves bounding the tail probability of $X$ from above and below as
\[
\P(X_n>t) \leq \P(X>t)\leq \P(X_n'>t),
\]
where both $(X_n)_{n\geq0}$ and $(X_n')_{n\geq0}$ satisfy \eqref{eq:seq} for some special $X_0$ and $X_0'$. 
Thus, understanding $\mathbb{P}(X>t)$ effectively reduces to analyzing the tail asymptotics of $\mathbb{P}(A Y + B > t)$, with $Y$ independent of $(A,B)$ but not necessarily satisfying \eqref{eq:affine}.

\subsection{Relation to the literature}
We note that similar ideas to the definition of $\LDM$ have appeared in related literature. In \cite[Theorem 1]{BDIM}, for instance, it is assumed that there exists a finite function $f$ such that 
\begin{align}\label{eq:BDIM}
f(y)=\lim_{t\to\infty} \frac{\P(Ay+B>t)}{\P(B>t)}.
\end{align}
If $A$ and $B$ are independent, $A$ has a finite moment generating function, and $t\mapsto \P(B> \log t)$ is regularly varying with index $-\alpha\leq 0$, then by the Breiman Lemma (see e.g.
\cite{Breiman07}) one obtains explicit form of $f(y) = \E[\exp(\alpha y A)]$. However, if $A$ and $B$ are not independent, yet \eqref{eq:BDIM} holds,
the form of $f$ may differ (see \cite[Remark 2.3]{BDIM}). 
Under \eqref{eq:BDIM} with $\P(B>t)\sim a\,t^c e^{-b t}$, $\E[e^{b B}\1_{A=1}]<1$ and some technical assumptions, \cite{BDIM} established that 
\[
\lim_{t\to\infty} \frac{\P(X>t)}{\P(B>t)}= \frac{ \E[f(X)]}{1-\E[e^{b B}\1_{A=1}]}.
\]

A closely related scenario is presented in \cite{DD18}, where it is assumed that $t\mapsto \P(A>t)$ is regularly varying with index $-\alpha<0$, along with the conditions $\E[A^\alpha]<1$ and $\limsup_{t\to\infty}\P(B>t)/\P(A>t)< \infty$. In this setting, $f$ is defined as
\[
f(y) = \lim_{t\to\infty} \frac{\P(Ay+B>t)}{\P(A>t)},
\]
and, under further technical conditions on the distribution of $A$, the tail asymptotics of $X$ are described by 
\[
\lim_{t\to\infty} \frac{\P(X>t)}{\P(A>t)} =\frac{ \E[f(X)]}{1-\E[A^\alpha]}. 
\]
In the special case of independence between $A$ and $B$, explicit computation of $f$ becomes possible.

Thus, our contribution can be viewed as an advancement of this line of research, relaxing independence assumptions and thereby enabling analysis of a broader class of models.

Additionally, logarithmic asymptotics of the tail behavior of $X$ were considered in \cite{BK18}, under the assumption that the functions
\begin{align*}
    t\mapsto -\log\P(1/(1-A)>t)\quad\mbox{and}\quad  t\mapsto -\log\P(B>t)
\end{align*}
are regularly varying (or when either $A$ or $B$ is bounded), again requiring independence of $A$ and $B$. For general dependent $A$ and $B$, however, only a lower bound on the tail was provided. 
As noted, the light tails heavily depend on the dependence structure of $(A,B)$. This relation seems to be captured quite well by a function $h$, which was defined in \cite{BK18} by 
\[
h(t) = \inf_{s\geq 1}\left\{
-s\log\P\left( \frac{1}{1-A}>s, B>\frac{t}{s}\right)
\right\},\qquad t>0. 
\]
There it was shown (see \cite[Theorem 5.1]{BK18}) that
\[
\liminf_{t\to\infty} \frac{\log\P(X>t)}{h(t)} \geq -c
\]
with an explicit constant $c$ depending on the index of regular variation of $h$. In Section \ref{sec:ex} we present an example demonstrating that this lower bound is generally not optimal.

In contrast, our paper addresses this gap by deriving precise logarithmic asymptotics within a general framework built upon the local dependence measure. Our results extend and complement previous findings in light-tailed scenarios where the tail behavior of $X$ is predominantly governed by the distribution of $B$ rather than $A$.

In \cite{BKT22} a less common problem in the theory of \eqref{eq:affine} was considered.  If $A$ and $B$ are nonnegative, then $X$, if exists, is also nonnegative. Thus, one may investigate the tail behavior of $X$ near the left endpoint of the support. 
 One of the main results in \cite{BKT22} establishes that if the limit
\[
g(y) = \lim_{\varepsilon\to 0^+} \frac{\log\P(A \varepsilon y+B<\varepsilon)}{\log\P(B<\varepsilon)}
\]
exists for $y\in[0,\infty)$, where $\varepsilon\mapsto -\log\P(B<\varepsilon)$ is a regularly varying function with index $\rho<0$ at $0$, then---under some technical assumptions (analogous to the admissibility condition in the present paper)---one has
\[
\lim_{\varepsilon\to 0^+} \frac{\log\P(X<\varepsilon)}{\log\P(B<\varepsilon)} = \inf_{y>1}\left\{ \frac{y^\rho}{y^\rho-1} g(y)\right\}. 
\]
Another result in \cite{BKT22} concerns the lower envelope of the sequence \eqref{eq:seq} with $X_0=0$; specifically, there exists an explicit deterministic scaling sequence $(h_n)_n$ such that
\[
\liminf_{n\to\infty} \frac{X_n}{h_n}
\]
is deterministic, positive and finite. 

There are clear analogies between the statements and proofs in the present paper and those in \cite{BKT22}. However, aside from Lemmas \ref{6.10.} and \ref{noweA.1}, which are taken directly from \cite{BKT22}, all other results in our paper have independent proofs. Notably, our analysis reveals quantitatively different behavior for $\rho\in(0,1]$ and $\rho>1$ (see, e.g., Theorem \ref{thm7.2.}), a phenomenon that was not observed in \cite{BKT22}. While \cite{BKT22} applied the left-tail results to a Fleming–Viot-type process, in Section \ref{sec:ex} we illustrate our findings using a new family of distributions. 

We also note that the assumptions in the present paper and in \cite{BKT22} are compatible.  Consequently, one can readily construct a distribution of $(A,B)$ such that both the left and right tail behaviors of $X$ satisfying \eqref{eq:affine} are asymptotically available and the  sequence $X_n=A_n X_{n-1}+B_n$  exhibits explicit lower and upper envelopes. For example, assume that the random pair $(A,B)$ is positively quadrant dependent (see Section \ref{sec:PQD}) and that for $\rho,\sigma>0$ one has
\[
\lim_{t\to\infty}\frac{-\log\P(B>t)}{t^{\rho}} = \lambda_+\qquad\mbox{and}\qquad \lim_{\varepsilon\to0^+}\frac{-\log\P(B<\varepsilon)}{\varepsilon^{-\sigma}}=\lambda_-
\]
together with
\[
a_-=\ess\inf(A)\geq 0\qquad\mbox{and}\qquad a_+=\ess\sup(A)< 1.
\]
Under these conditions, since $\E[\log A]<0$, $\E[\max\{\log B,0\}]<\infty$, 
Theorems \ref{theorem4.1} and \ref{thm7.2.} in this paper together with \cite[Theorem 4.1 and Proposition 5.4]{BKT22} yield
\[
\lim_{t\to\infty}\frac{-\log\P(X>t)}{t^{\rho}} = c_+\qquad\mbox{and}\qquad \lim_{\varepsilon\to0^+}\frac{-\log\P(X<\varepsilon)}{\varepsilon^{-\sigma}}=c_-,
\]
where
\[
c_+ = \lambda_+ \begin{cases}
    \left(1-a_+^{\rho/(\rho-1)}\right)^{\rho-1}, & \mbox{if }\rho>1, \\
    1, & \mbox{if }\rho\in(0,1),
\end{cases} \qquad\mbox{and}\qquad c_- = \lambda_-\left(1- a_-^{\sigma/(1+\sigma)}\right)^{-(1+\sigma)}.
\]
Moreover, if $X_0=0$, then by Theorem \ref{nowe6.1} and \cite[Theorem 6.1]{BKT22}, 
\[
\limsup_{n\to\infty} \frac{X_n}{(\log n)^{1/\rho}}=c_+^{-1/\rho} \qquad \mbox{and}\qquad\liminf_{n\to\infty}  \frac{X_n}{(\log n)^{-1/\sigma}} =c_-^{1/\sigma}.
\]

\subsection{Organization of the paper}
The remainder of the paper is organized as follows. In the next section, we give a short review of the theory of regularly varying functions. Section \ref{sec:LDM} introduces the local dependence measure and its Legendre-type transform, and establishes their key properties. We show that when $A$ and $B$ are independent (or more generally, positively quadrant dependent), the LDM admits an explicit representation. In Section \ref{sec:AX+B}, we derive the asymptotics of $t\mapsto\log \P(AX+B>t)$ under independence between $X$ and $(A,B)$, where $X$ does not necessarily satisfy \eqref{eq:affine}. This is a result that is instrumental in controlling the asymptotics of $\P(X_n>t)$. Section \ref{sec:X=AX+B} discusses the existence, uniqueness, and basic properties of solutions to \eqref{eq:affine}. Section \ref{sec:tailmain} is devoted to our first main result, the precise logarithmic tail asymptotics of the solution to \eqref{eq:affine}. In Section \ref{sec:UpperEnvelope} we prove our second main result: the identification of the upper envelope for the sequence $(X_n)_{n\geq 0}$ with $X_0=0$. Finally, in Section \ref{sec:ex} we describe a family of distributions of pairs $(A,B)$ for which the LDM is computable and provide its general form. We use our earlier results to demonstrate that the lower bound in \cite{BK18} is not optimal in general.

\section{Regular variation}
We write $f(x)\sim g(x)$ if $\lim_{x\to\infty} f(x)/g(x)=1$.
A measurable function $f\colon (0, \infty) \rightarrow (0,\infty)$ is said to be regularly varying of index $\rho\in\R$ if and only if $f(\lambda x)\sim \lambda^\rho f(x)$ for all $\lambda > 0$.    

The class of regularly varying functions of index $\rho \in \R$ is denoted by $\mathcal{R}^{\rho}$. In particular, elements of the class $\mathcal{R}^{0}$ are called slowly varying functions.

We  gather all the necessary properties of regularly varying functions in the lemma below.
\begin{lemma}\label{granica}
Let $f\in\mathcal{R}^\rho$ with $\rho>0$. Then:
\begin{enumerate}
    \item[(i)] $f(x) = x^{\rho}\ell(x)$ for some slowly varying function $\ell\in\mathcal{R}^0$.
    \item[(ii)] $\lim_{x \to\infty}f(x)= \infty$.
    \item[(iii)] For any $C>1$ and $\delta >0$, there exists $K$ such that for $x, y \geq K$, 
\[
\frac{f(y)}{f(x)} \leq C \max \left\{\left(\frac{y}{x}\right)^{\rho+\delta}, \left(\frac{y}{x}\right)^{\rho-\delta}\right\}.
\]
    \item[(iv)] There exists $g\in \mathcal{R}^{1/\rho}$ such that 
    \[
    f(g(x)) \sim g(f(x))\sim x
    \]
    and $g$ is determined uniquely to within asymptotic equivalence $\sim$. 
\item[(v)]  There exists $\tilde{f} \in \mathcal{R}^{\rho}$ such that $\tilde{f}$ continuous, strictly increasing and $\tilde{f}(x) \sim f(x)$.
\end{enumerate}
\end{lemma}
\begin{proof}
    (i) follows directly from the definition of $\mathcal{R}^\rho$. 
(ii) is proved in \cite[Proposition 1.5.1]{BGT89}, 
    (iii) is known as the Potter bounds, \cite[Theorem 1.5.6]{BGT89}. 
    (iv) is proved in \cite[Theorem 1.5.12]{BGT89}.
    For (v), fix $X>0$ so large that $f$ is locally bounded on [X, $\infty$) and define $\tilde{f}(x) = \rho \int_X^x t^{-1}f(t)\dd t$ for $x\in(X,\infty)$. Extend the definition of $\tilde{f}$ to $(0,X]$ in a way that it is continuous and strictly increasing on $(0,\infty)$. By Karamata's Theorem, \cite[Theorem 1.5.11]{BGT89}, we have $\tilde{f}(x)\sim f(x)$. 
\end{proof}

\section{Exponential decay and local dependence measure}\label{sec:LDM}

\begin{defin}
Let $f$ be a function defined on a neighborhood of infinity such that $f(t)\to\infty$ as $t\to\infty$. 
We say that a nonnegative random variable $X$ has an exponential-$f$-decay tail if  
\begin{align*}
\lim_{t\to\infty} \frac{-\log\P(X > t)}{f(t)} = \lambda,
\end{align*}
where $\lambda\in[0,\infty]$. We call such a random variable an $\ED_{f}(\lambda)$-random variable.
\end{defin}

We note that if $X$ is a bounded random variable, then it is $\ED_f(\infty)$ for any $f$. 

\begin{defin}
    Let $(A, B)$ be a pair of nonnegative random variables, and let $f\in\mathcal{R}^\rho$ with $\rho>0$. A function $g\colon [0, \infty) \rightarrow [0, \infty]$ is said to be the local dependence measure of $(A,B)$ (denoted $\LDM_f^\rho$), if the limit
\[
g(y)=\lim_{t\to\infty} \frac{-\log\P(A ty + B > t)}{f(t)},\qquad y\in[0,\infty)
\]
exists.
\end{defin}

In general, for an arbitrary pair $(A,B)$ and a function $f$, the existence of $g$ on $[0,\infty)$ is not guaranteed. Nevertheless,  we will show that $g(y) = 0$ for $y>a_+^{-1}$, where 
\[
a_+ = \ess\sup(A).
\]
Furthermore, if $A$ and $B$ are positively quadrant dependent (PQD) and $g(0)<\infty$, then $g$ exists on $[0,\infty)$ and its explicit form can be derived---it depends only on $a_+$ and $g(0)$ ( see Section \ref{sec:PQD} for details). Another family of distributions for which an explicit expression for $g$ is available is presented in Section \ref{sec:ex}. 

Assume that $\tilde{f} \in \mathcal{R}^{\rho}$ is such that $\tilde{f}$ continuous, strictly increasing and $\tilde{f}(x) \sim f(x)$.
Clearly, we have $\ED_f(\lambda) = \ED_{\tilde{f}}(\lambda)$ and $\LDM_f^\rho = \LDM_{\tilde{f}}^\rho$. In view of Lemma \ref{granica}, without loss of generality, we will assume from now on that every regularly varying function $f$ is continuous and strictly increasing.

\begin{lemma}\label{3}
Assume that $(A,B)$ are nonnegative random variables and $f \in \mathcal{R}^{\rho}$ with $\rho>0$.
Suppose $g(0)$ exists and $a_+>0$. Then:
\begin{enumerate}
\item[(i)] For every $y>a_+^{-1}$,  \[
\lim_{t\to\infty} \frac{-\log\P(A ty + B > t)}{f(t)}=0.
\]
\item[(ii)] For every $y\in[0,a_+^{-1})$,
\[
\liminf_{t\to\infty} \frac{-\log\P(A ty + B > t)}{f(t)}\geq g(0)(1-a_+ y)^\rho.
\]
\item[(iii)] If $y > x \geq 0$ and  both $g(x)$ and $g(y)$ exists,  then $g(y)\leq g(x)$, 
\end{enumerate}
\end{lemma}
\begin{proof}
\begin{enumerate} 
        \item[(i)] Assume $y>a_+^{-1}$. Clearly, $\P(A>1/y)>0$. Notice that $\{A>1/y\} = \{A t y>t\}\subset\{A t y+B>t\}$ for $t>0$ and therefore, 
        \begin{multline*}
            0 \leq \liminf_{t\to\infty} \frac{-\log\P(A ty + B > t)}{f(t)}\\ \leq \limsup_{t\to\infty} \frac{-\log\P(A ty + B > t)}{f(t)} 
             \leq \lim_{t\to\infty} \frac{-\log\P(A > 1/y)}{f(t)}=0.
        \end{multline*}
        \item[(ii)] Since $\P(A \leq a_+) = 1$, we have
        \begin{align*}
            \P(Aty + B >t) \leq \P(a_+ty+B>t) = \P(B>t(1-a_+y)) .
        \end{align*}
        Thus, for $y \in [0, 1/a_+)$, 
        \begin{align*}
            \liminf_{t\to\infty} \frac{-\log\P(A ty + B > t)}{f(t)} 
            &\geq  \lim_{t \to\infty} \frac{-\log \P(B>t(1-a_+y))}{f(t)}  \\
            & =
            \lim_{t \to\infty} \frac{-\log \P(B>t(1-a_+y))}{f(t(1-a_+y))}\cdot \frac{f(t(1-a_+y))}{f(t)} \\
            &= g(0)(1-a_+y)^{\rho}.
        \end{align*}
        \item[(iii)] Let $y>x\geq 0$. We note that since $A$ is a nonnegative random variable, for all $t>0$, we have the inclusion $\{ Atx + B > t \} \subset \{ Aty + B > t\}$. By the monotonicity of $\P$ and the assumption that both $g(x)$ and $g(y)$ are well defined we get that $g(y) \leq g(x)$.
\end{enumerate}
\end{proof}
\begin{remark}\label{uwaga_do_3}
Lemma \ref{3} implies several key properties of the local dependence measure $g$: 
\begin{itemize} 
\item On the interval $(a_+^{-1},\infty)$, the limit always exists and equals $0$. 
\item If $g$ is well defined on $[0,a_+^{-1})$, then it is bounded below by $g(0)(1-a_+y)^\rho$ on this interval. 
\item If $g$ exists on the entire $[0,\infty)$, then it is a nonincreasing function. \end{itemize} 
In Section \ref{sec:PQD}, we show that the lower bound is attained for positively quadrant dependent $A$ and $B$. 
\end{remark}

\begin{defin}
For a function $g\colon (0,\infty) \to [0,\infty]$, let $\phi_\rho\colon[0,\infty)\to[0,\infty]$ be a Legendre-type transform defined by 
\[
\phi_{\rho}(\lambda) = \inf_{y > 0} \left\{y^{\rho}\lambda+g(y)\right\},\qquad \lambda\in[0,\infty).
\]
\end{defin}

From this point forward, unless explicitly stated otherwise, we assume that $g$ exists on $[0,\infty)$.

\begin{lemma}\label{lem:phi-prop}\ 
Let $g$ be the $\LDM_f^\rho$ for the nonnegative random variables $(A,B)$ such that $a_+=\ess\sup(A)>0$, where $f \in \mathcal{R}^{\rho}$ and $\rho>0$. Let $\phi_\rho$ be its Legendre-type transform. Then:
\begin{enumerate}
\item[(i)] $\phi_\rho$ is finite, nondecreasing, concave. 
\item[(ii)] $\phi_\rho(0)=\phi_\rho(0^+)=0$.
\item[(iii)] $\phi_\rho$ is continuous.
\item[(iv)] $\sup_{y>0}\left\{g(y)\right\} 
    \geq \phi_{\rho}(\lambda)$.
\item[(v)] If there exists $\lambda > 0$ such that $\phi_{\rho}(\lambda) > \lambda$, then $\phi_{\rho}(\lambda) = \lambda$ for at most one $\lambda > 0$. 
\item[(vi)] With $\gamma=g(0)$ and  $a_+=\ess\sup(A)>0$,
    \begin{align*}
\lambda a^{-\rho}\geq \phi_{\rho}(\lambda) \geq  
\begin{cases}
\left(\gamma^{1/(1-\rho)}+a_+^{\rho/(\rho-1)}\lambda^{1/(1-\rho)}\right)^{1-\rho},&  \mbox{ if } \rho>1,\\
    \min\left\{\gamma, \lambda a_+^{-\rho}\right\},              &  \mbox{ if } 0<\rho \leq 1.
\end{cases}
\end{align*}
\end{enumerate}
\end{lemma}
\begin{proof}
\begin{enumerate}
\item[(i)] Finiteness follows from Lemma \ref{3} (i), monotonicity is obvious, $\phi_\rho$ as the point-wise infimum of a family of affine functions is concave.
\item[(ii)] We have $\phi_\rho(0) = \inf_{y > 0} \left\{g(y)\right\}=0$  by Lemma \ref{3} (i). Moreover, by the definition of $\phi_{\rho}$ we get 
$\phi_{\rho}(\lambda) \leq y^{\rho}\lambda + g(y)$ for any $y>0$. Thus, by taking $y$ sufficiently large (recall Lemma \ref{3} (i)) and letting $\lambda\downarrow0^+$, we obtain $\phi_{\rho}(0^+)\leq0$. 
\item[(iii)] From (ii) we get continuity in $0$. Continuity on $(0, \infty)$ is a consequence of the fact that $\phi_{\rho}$ is concave, which was stated in (i).
\item[(iv)] By the definition of $\phi_{\rho}(\lambda)$, we have $\phi_{\rho}(\lambda) \leq y^{\rho}\lambda+g(y)$ for all $y>0$. Letting $y \downarrow 0^{+}$ (noting that $g$ is a nonincreasing function, so its limit as $y\to 0^{+}$ exists, although it may be infinite), we obtain $\phi_{\rho}(\lambda) \leq g(0^{+}) = \sup_{y>0}\{g(y)\}$.
\item[(v)] Define $\psi(\lambda)=\phi_\rho(\lambda)-\lambda$. Then $\psi$ is concave, continuous, and $\psi(0)=0$. Suppose there exist $0<\lambda_1<\lambda_2$ with $\phi_\rho(\lambda_1)=\lambda_1$ and $\phi_\rho(\lambda_2)=\lambda_2$, so that $\psi(0)=\psi(\lambda_1)=\psi(\lambda_2)=0$. By concavity, $\psi\equiv0$ on $[0,\lambda_2]$ and $\psi(\lambda)\leq 0$ for $\lambda>\lambda_2$. Now, if some $\lambda_0$ satisfies $\phi_\rho(\lambda_0)>\lambda_0$ (i.e., $\psi(\lambda_0)>0$), then we obtain a contradiction. Hence, the equation $\phi_\rho(\lambda)=\lambda$ can have at most one solution for $\lambda>0$.
\item[(vi)] The upper and lower bounds on $\phi_\rho(\lambda)$ are based on properties (i) and (ii) from Lemma \ref{3}. The derivation of these bounds follows from the proof of Theorem \ref{thm7.2.}.
\end{enumerate}
\end{proof}

\begin{defin}
    We say that a $\LDM_f^\rho$ function $g$ is admissible if either:
there exists $\lambda>0$ such that $\phi_\rho(\lambda)>\lambda$ or $g(0)=0$ (in which case $\phi_\rho\equiv 0$).
\end{defin}

The notion of admissibility is introduced to exclude those distributions of $(A,B)$ and choices of the function $f$ for which our analytical framework fails (see Lemma \ref{lemma4.3}). To illustrate, consider the following example. Assume that the nonnegative random variables $A$ and $B$ are independent, $a_+=\ess\sup(A)\in(0,1]$ and $f(t)=-\log\P(B>t)$ is regularly varying with index $-\rho<0$. Then, by Theorem \ref{thm7.2.}, $g$ exists and is given by $g(y) = \max\{1-a_+ y,0\}^\rho$ for $y\geq 0$. In this case, one can verify that the admissibility condition holds.

In contrast, consider the same model but with the choice  $f(t)= t^{\sigma}$ for some $\sigma\in(0,\rho)$. In this setting, the $\LDM$ is given by  
\[
g(y)=\lim_{t\to\infty} \frac{-\log\P(A ty + B > t)}{t^\sigma}= \begin{cases}
\infty, & \mbox{ if }y< a_+^{-1},\\
0, & \mbox{ if }y>a_+^{-1},
\end{cases}
\]
for any $\sigma\in(0,\rho)$. A direct calculation shows that the corresponding Legendre-type transform satisfies  $\phi_\sigma(\lambda)= \lambda$ for all $\lambda>0$, thereby violating the condition  $\phi_\sigma(\lambda)>\lambda$.   Consequently, this LDM is non-admissible.

\begin{defin}
For a $\LDM_f^\rho$ function $g$, we define
\begin{align}\label{eq:lambdastar}
\lambda^\ast = \inf_{y\in(0,1)}\left\{\frac{g(y)}{1-y^{\rho}}\right\}.
\end{align}
\end{defin}

\begin{lemma}\label{nowe3.15}\ 
\begin{enumerate}
    \item[(i)]Assume $c \geq 0$. Then, $\phi_{\rho}(c) \geq c$ if and only if $c \leq \lambda^\ast$. 
    \item[(ii)] If $\lambda^\ast<\infty$, then $\phi_\rho(\lambda^\ast)=\lambda^\ast$.
    \item[(iii)]  With $\gamma=g(0)$ and  $a_+=\ess\sup(A)>0$,
    \[
\lambda^\ast \geq 
\begin{cases}
 \gamma\left(1-a_+^{\rho/(\rho-1)}\right)^{\rho - 1},&  \mbox{ if } \rho>1,\\
    \gamma,              &  \mbox{ if } 0<\rho \leq 1.
\end{cases}    \]
\item[(iv)]  $\lambda^\ast=\infty\iff g(0)=\infty$.
\end{enumerate}
\end{lemma}
\begin{proof}
\begin{enumerate}
\item[(i)] We have 
\begin{align*}
\phi_{\rho}(c) \geq c &\iff \inf_{y>0}\left\{y^{\rho}c + g(y)\right\} \geq c \iff \forall\, y > 0 \,\,\,  y^{\rho}c + g(y) \geq c \\
&\iff \forall\, y\in(0,1)\,\,\, y^{\rho}c + g(y) \geq c  \iff c \leq \inf_{y\in(0,1)}\left\{\frac{g(y)}{1-y^{\rho}}\right\}.
\end{align*}
\item[(ii)] Let $(\lambda_n)_n$ be a sequence such that $\lambda_n \downarrow \lambda^\ast$. Then, by (i) we have $\phi_\rho(\lambda_n)<\lambda_n$ and by continuity of $\phi_\rho$, we obtain $\phi_\rho(\lambda^\ast)\leq \lambda^\ast$. Setting $c=\lambda^\ast$ in (i), we obtain reversed bound.
\item[(iii)] The lower bound on $\lambda^\ast$ quickly follows from the lower bound on $g$ in Lemma \ref{3} (ii). 
\item[(iv)] If $g(0)=\infty$, then by (iii) we obtain $\lambda^\ast\geq \infty$. If $\lambda^\ast=\infty$, then by definition of $\lambda^\ast$, we have for any $y\in(0,1)$, $g(y)\geq (1-y^\rho)\lambda^\ast$. Since $g$ is nonincreasing, $g(0)=\infty$. 
\end{enumerate}
\end{proof}

\subsection{Positive quadrant dependence}\label{sec:PQD}
\begin{defin}
We say that random variables $A$ and $B$ are positively quadrant dependent (PQD for short) if, for all $a,b \in \R$, 
\begin{align*}
\P(A>a, B>b) \geq \P(A > a) \P(B>b). 
\end{align*}
\end{defin}

It turns out that for a PQD pair $(A,B)$, the $\LDM$ always exists $[0,\infty)$ and, moreover, has explicit form. 

Recall that $a_+ = \ess \sup(A) = \inf \left\{x \in \R\colon \P(A>x)=0 \right\}$.
\begin{thm}\label{thm7.2.}
Assume that nonnegative random variables $A$ and $B$ are PQD. Suppose that $\gamma:=g(0)$ exists and is finite for some $f\in\mathcal{R}^\rho$ with $\rho>0$. Then, the local dependence measure $g$ of $(A,B)$ exists on $[0,\infty)$. Suppose $a_+>0$.
\begin{enumerate}
    \item[(i)] For $y \geq 0$,
\begin{align*}
g(y)= \gamma \max\{1-a_+ y,0\}^\rho.
\end{align*} 
    \item[(ii)] If $a_+\in(0,1]$, then for $\lambda \geq 0$, 
\begin{align*}
\phi_{\rho}(\lambda) &= 
\begin{cases}
\left(\gamma^{1/(1-\rho)}+a_+^{\rho/(\rho-1)}\lambda^{1/(1-\rho)}\right)^{1-\rho},&  \mbox{ if } \rho>1,\\
    \min\left\{\gamma, \lambda a_+^{-\rho}\right\},              &  \mbox{ if } 0<\rho \leq 1,
\end{cases}
\intertext{and}%\qquad\mbox{and}\qquad
\lambda^\ast &=
\begin{cases}
   \gamma\left(1-a_+^{\rho/(\rho-1)}\right)^{\rho - 1},&  \mbox{ if } \rho>1,\\
    \gamma,              &  \mbox{ if } 0<\rho \leq 1.
\end{cases}
\end{align*}
\end{enumerate}
\end{thm}
\begin{proof}
\begin{enumerate}
    \item[(i)] By Lemma \ref{3} (ii), for $y \in [0, a_+^{-1})$ we have 
    \[
\liminf_{t\to\infty} \frac{-\log\P(Aty+B>t)}{f(t)}  \geq \gamma (1-a_+y)^\rho.
    \]
Since we assume that $g(0)$ is finite, the above inequality extends to  $y= a_+^{-1}$. 

We now show that for $y \in [0, a_+^{-1}]$, the quantity $\gamma (1-a_+y)^\rho$ is an upper bound for the superior limit.
Fix $\delta \in (0, a_+)$. Then, if $0\leq y \leq a_+^{-1} < (a_+-\delta)^{-1}$, we obtain
\begin{multline*}
\P(Aty+B>t) \geq \P(Aty+B>t, A >a_+-\delta)  \geq  \P((a_+-\delta)ty+B>t, A>a_+-\delta) \\
\geq \P(A>a_+-\delta)\, \P(B>t(1-(a_+-\delta)y)),
\end{multline*}
where we have used the fact that $A$ and $B$ are PQD. 
Therefore,
\begin{multline*}
\limsup_{t\to\infty} \frac{-\log\P(Aty+B>t)}{f(t)} \leq \lim_{t \to\infty} \frac{-\log \P(A>a_+-\delta) - \log \P(B>t(1-(a_+-\delta)y)) }{f(t)}
\end{multline*}
Since $f(t)\to\infty$ as $t\to\infty$, and from the fact that $\P(A>a_+-\delta) > 0$, we obtain 
\begin{multline*}
\limsup_{t\to\infty} \frac{-\log\P(Aty+B>t)}{f(t)}   \\ \leq \lim_{t \to\infty} \frac{- \log \P(B>t(1-(a_+-\delta)y)) }{f(t(1-(a_+-\delta)y))} \cdot \frac{f(t(1-(a_+-\delta)y))}{f(t)}  = \gamma(1-(a_+-\delta)y)^\rho.
\end{multline*}
By letting $\delta\downarrow0^+$, we conclude that for $0\leq y\leq a_+^{-1}$,
\begin{multline*}
\gamma(1-a_+ y)^\rho \leq \liminf_{t\to\infty} \frac{-\log\P(Aty+B>t)}{f(t)} \\
\leq \limsup_{t\to\infty} \frac{-\log\P(Aty+B>t)}{f(t)} \leq  \gamma(1-a_+y)^{\rho}.
\end{multline*}
Thus, $g$ exists on $[0, a_+^{-1}]$ and is equal to $\gamma(1-a_+y)^{\rho}$ on this interval. Finally, by Remark \ref{uwaga_do_3}, for $y >a_+^{-1}$ we have that $g(y)$ exists and equals $0$. This completes the proof of (i).
\item[(ii)] 
Since $g(y)=0$ for $y\geq a_+^{-1}$, we have
\begin{align*}
    \phi_{\rho}(\lambda) &= \inf_{y>0}\left\{\lambda y^\rho + g(y) \right\} = \inf_{y\in(0,a_+^{-1})}\left\{\lambda y^\rho + g(y) \right\}=\inf_{y\in(0,a_+^{-1})}\left\{h_1(y)\right\},
\end{align*}
where $h_1\colon(0,a_+^{-1})\to\R$ is defined by $h_1(y) = \lambda y^\rho + \gamma (1-a_+y)^\rho$.

If $0<\rho < 1$, then $h_1$ is increasing on $(0, K)$ and decreasing on $(K,a_+^{-1})$, where $K= (a_++(\tfrac{a_+\gamma}{\lambda})^{1/(1-\rho)})^{-1}$. Therefore,
\begin{align*}
    \inf_{(0, a_+^{-1})}\left\{h_1(y)\right\} = \min\left\{ h_1(0^+), h_1((a_+^{-1})^-)\right\} = \min\left\{ \gamma, \lambda a_+^{-\rho}\right\}.
\end{align*}
The case $\rho \in (0, 1)$ easily extends to $\rho = 1$. 

If $\rho>1$, then $h_1$ is decreasing on $(0, K)$ and increasing on $(K, a_+^{-1})$, where $K$ is as before. 
Therefore,
\begin{align*}
    \inf_{(0, a_+^{-1})}\left\{h_1(y)\right\} = h_1(K) = \frac{\lambda \gamma}{\left(\lambda^{\frac{1}{\rho - 1}}+a_+^{\frac{\rho}{\rho - 1}}\gamma^{\frac{1}{\rho -1}}\right)^{\rho-1}} = \left(\gamma^{1/(1-\rho)}+a_+^{\rho/(\rho-1)}\lambda^{1/(1-\rho)}\right)^{1-\rho}.
\end{align*}
Moreover, it is easy to see that $h_1(K)\leq \lambda a_+^{-\rho}$. 

We have
\begin{align*}
    \lambda^\ast = \inf_{y \in (0,1)}\left\{\frac{g(y)}{1-y^\rho}\right\}=\inf_{y \in (0,1)}\left\{h_2(y)\right\},
\end{align*}
where $h_2\colon(0,1) \to\R$ is defined by $h_2(y) = \gamma (1-a_+y)^\rho/(1-y^\rho)$.
If $\rho\in(0,1]$, then $h_2$ is increasing on $(0,1)$ and therefore
\begin{align*}
    \inf_{y \in (0,1)}\left\{h_2(y)\right\} = h_2(0^+) = \gamma.
\end{align*}
If $\rho > 1$, then $h_2$ is increasing on $(a_+^{1/(\rho-1)}, 1)$ and decreasing on $(0, a_+^{1/(\rho-1)})$. Therefore,
\begin{align*}
    \inf_{y \in (0,1)}\left\{h_2(y)\right\} = h_2\left(a_+^{1/(\rho-1)}\right) = \gamma \left(1-a_+^{\rho/(\rho-1)}\right)^{\rho - 1}.
\end{align*}
\end{enumerate}
\end{proof}

\section{\texorpdfstring{Tails of $AX+B$}{Tails of AX+B}}\label{sec:AX+B}
Recall our standing assumption that $\P(A>0)>0$. In this section, $X$ is assumed to be independent of $(A,B)$ but does not necessarily satisfy \eqref{eq:affine}. 

\begin{thm}\label{nowe3.9}
Let $g$ be the $\LDM_f^\rho$ for the nonnegative random variables $(A,B)$, where $f \in \mathcal{R}^{\rho}$ with $\rho>0$. 
 Suppose $X$ is $\ED_{f}(\lambda)$ with $\lambda\in[0,\infty)$. Then $AX+B$ is $\ED_f(\phi_\rho(\lambda))$.
\end{thm}

\begin{proof}
First, we prove that 
\begin{align}\label{eq:upperbound}
\limsup_{t\to\infty} \frac{-\log\P(AX + B > t)}{f(t)} \leq \phi_{\rho}(\lambda).
\end{align}
For any $y>0$ and $t>0$, we have
\begin{align*}
\P(AX+B>t) &\geq \P(AX+B>t, X>ty) \geq \P(Aty+B>t, X > ty) \\
& = \P(Aty+B>t)\,\P(X > ty).
\end{align*}
Thus, 
\begin{align*}
    \limsup_{t \to\infty}{\frac{-\log\P(AX+B>t)}{f(t)}} &\leq \limsup_{t \to\infty}{\frac{-\log\P(Aty+B>t) - \log\P(X > ty)}{f(t)}} \\
%    & \leq \limsup_{t \to\infty}{\frac{-\log\P(Aty+B>t)}{f(t)}} + \limsup_{t \to\infty}{\frac{-\log\P(X>ty)}{f(t)}} \\
    &\leq g(y) + \limsup_{t \to\infty}{\frac{-\log\P(X>ty)}{f(ty)}\cdot \frac{f(ty)}{f(t)}} = g(y) + \lambda y^{\rho}.
\end{align*}
By taking $\inf_{y>0}$ on both sides, we obtain \eqref{eq:upperbound}. 

Next, we establish the lower bound,
\begin{align}\label{eq:lower}
\liminf_{t\to\infty} \frac{-\log\P(AX + B > t)}{f(t)} \geq \phi_{\rho}(\lambda).
\end{align}
If $\lambda=0$, then by Lemma \ref{lem:phi-prop} (ii), we have $\phi_\rho(\lambda)=0$, and the above inequality is trivial. Now suppose $\lambda>0$. Since $\phi_\rho$ is finite, there exists $a > 0$ such that $\lambda a^{\rho} \geq \phi_{\rho}(\lambda)$. We note that 
\begin{align}
\begin{split}
\liminf_{t\to\infty} &\frac{-\log\P(AX + B > t, X > ta)}{f(t)} \geq \liminf_{t\to\infty} \frac{-\log\P(X > ta)}{f(t)} \\
&= \liminf_{t\to\infty} \frac{-\log\P(X > ta)}{f(ta)} \cdot \frac{f(ta)}{f(t)} = \lambda a^{\rho} \geq \phi_{\rho}(\lambda).
\end{split}
\label{11}
\end{align}
Moreover, by Lemma \ref{lem:phi-prop} (iv), we have $\sup_{y>0}\{g(y)\} \geq \phi_{\rho}(\lambda)$. 
Fix $\eta > 0$. There exists $b\in(0,a)$ such that $g(b) \geq \phi_{\rho}(\lambda) - \eta/2$. 
We have
\begin{align}\label{10}
\begin{split}
    \liminf_{t\to\infty} \frac{-\log\P(AX + B > t, X \leq tb)}{f(t)} & \geq
\liminf_{t\to\infty} \frac{-\log\P(Atb + B > t)}{f(t)}  = g(b) \\ & \geq \phi_{\rho}(\lambda) - \frac{\eta}{2}.
\end{split}
\end{align}
Combining \eqref{11} and \eqref{10}, we deduce that there exists $M>0$ such that for all $t>M$,
\begin{align*}
\P(AX+B>t, X>ta) +  \P(AX+B>t, X \leq tb)  \leq 2 \exp(-f(t)(\phi_{\rho}(\lambda) - \eta)).
\end{align*}
Now, fix $y, h > 0$ such that $h<y$. Then,
\begin{align}\label{4}\begin{split}
\P(AX+B>t, t(y-h) < X \leq ty) &\leq \P(Aty+B>t,t(y-h) < X) \\
& =
\P(Aty + B > t)\,\P(X>t(y-h)). \end{split}
\end{align}
Since
\begin{align*}
\lim_{t \rightarrow \infty}{\frac{-\log \P(X > t(y-h))}{f(t)}} = 
\lim_{t \rightarrow \infty}{\frac{-\log \P(X > t(y-h))}{f(t(y-h))}} \cdot
\frac{f(t(y-h))}{f(t)} = \lambda(y-h)^{\rho},
\end{align*}
for sufficiently large $t$, we obtain
\begin{align*}
\P(X>t(y-h)) \leq \exp(-f(t)(\lambda(y-h)^{\rho} - \eta)).
\end{align*}
If $g(y) < \infty$, then, from the definition of $g$, we conclude that 
\begin{align*}
\exists\, M > 0 \mbox{ such that } \forall\, t>M \quad \P(Aty + B>t) \leq \exp(-f(t)(g(y) - \eta)). 
\end{align*}
Thus, from \eqref{4}, we get for sufficiently large $t$,
\begin{align}\label{12}
\begin{split}
\P(AX+B>t, \text{ }t(y-h) < X \leq ty) \leq 
\exp(-f(t)(g(y)+\lambda(y-h)^{\rho}-2\eta)).
%\\= \exp(-f(t)(g(y)+\lambda y^{\rho}-\lambda y^{\rho} + \lambda(y-h)^{\rho}-2\eta)) 
\end{split}
\end{align}
If $g(y)=\infty$, then for any $G>0$ there exists $N>0$ such that for all $t>N$, we obtain a bound 
\begin{align}\label{12'}
\P(AX+B>t, \text{ }t(y-h) < X \leq ty) \leq 
\exp(-f(t)(G+\lambda(y-h)^{\rho})).
\end{align}

We now consider two cases:
(i) $\rho \geq 1$ and (ii) $\rho\in(0,1)$. 

(i) By the convexity of $(0,\infty)\ni x\mapsto x^\rho$, we obtain 
\[
\lambda y^{\rho} - \lambda (y-h)^{\rho} \leq \lambda h \rho y^{\rho -1}.
\]
Using \eqref{12}, the definition of $\phi_\rho$, and the above inequality, we conclude that if $g(y)<\infty$, then for all $y>h>0$ and  sufficiently large $t$,
\begin{align*}
\P(AX+B>t, t(y-h) < X \leq ty) &\leq \exp(-f(t)(g(y)+\lambda y^{\rho}-\lambda y^{\rho} + \lambda(y-h)^{\rho}-2\eta)) \\ 
&\leq \exp(-f(t)(\phi_{\rho}(\lambda) - \lambda h \rho y^{\rho -1} - 2\eta)). 
\end{align*}
The same bound holds when $g(y)=\infty$ by choosing $G$ sufficiently large in \eqref{12'}. 

Fix $n\in\N$.
Let $h_0 =0$, $h_k - h_{k-1} = \frac{a-b}{n}=h$ for $k=1,\ldots,n$. Then for sufficiently large $t$, 
\begin{align*}
\P(AX+B>t, tb < X \leq ta) &= \sum_{k=1}^{n}{\P(AX+B>t, t(a-h_k)<X \leq t(a-h_{k-1}))} \\
&\leq
\sum_{k=1}^{n}{\exp(-f(t)(\phi_{\rho}(\lambda)-\lambda h \rho (a-h_{k-1})^{\rho -1}-2\eta))} \\
&\leq
n \exp(-f(t)(\phi_{\rho}(\lambda)-\lambda h \rho a^{\rho -1}-2\eta)). 
\end{align*}
Thus, for sufficiently large $t$,
\begin{align*}
\P(AX+B>t) &= \P(AX+B>t, X\in t(b,a]) + \P(AX+B>t, X\notin t(b,a]) \\
&\leq 
n \exp(-f(t)(\phi_{\rho}(\lambda)-\lambda h \rho a^{\rho -1}-2\eta)) + 2 \exp(-f(t)(\phi_{\rho}(\lambda)-\eta)) \\ 
&\leq 
(n+2){\exp(-f(t)(\phi_{\rho}(\lambda)-\lambda h \rho a^{\rho -1}-2\eta))}. 
\end{align*}
Therefore, for sufficiently large $t$,
\begin{align*}
\frac{-\log \P(AX+B>t)}{f(t)} \geq \frac{-\log(n+2)}{f(t)}+\phi_{\rho}(\lambda)-\lambda h \rho a^{\rho -1} - 2\eta. \\
\end{align*}
Taking $\liminf_{t \rightarrow \infty}$ on both sides, we get
\begin{align*}
\liminf_{t \rightarrow \infty}\frac{-\log \P(AX+B>t)}{f(t)} \geq \phi_{\rho}(\lambda)-\lambda h \rho a^{\rho -1} - 2\eta. 
\end{align*}
By letting $n \uparrow \infty$ (recall that $h = (a-b)/n$) and then $\eta \downarrow 0^{+}$, we obtain \eqref{eq:lower}.

(ii) For $\rho\in(0,1)$, we have $y^{\rho}-(y-h)^{\rho} \leq h^{\rho}$ for $y>h>0$. 
Similarly to case (i), for all $y>h>0$ and sufficiently large $t$, we get
\begin{align*}
\P(AX+B>t, t(y-h)<X \leq ty) & \leq \exp(-f(t)(g(y)+\lambda y^{\rho}-\lambda y^{\rho} + \lambda(y-h)^{\rho}-2\eta)) \\ 
& \leq \exp(-f(t)(\phi_{\rho}(\lambda)-\lambda h^{\rho}-2\eta)) 
\end{align*}
and this bound also holds when $g(y)=\infty$. 
Proceeding as in (i), 
we obtain that for sufficiently large $t$,
\begin{align*}
\P(AX+B>t) \leq (n+2){\exp(-f(t)(\phi_{\rho}(\lambda)-\lambda h ^{\rho} -2\eta))}.
\end{align*}
The remaining part of the proof is analogous.
\end{proof}

\section{\texorpdfstring{The equation $X\stackrel{d}{=}AX+B$}{The equation X=AX+B}}\label{sec:X=AX+B}

The stochastic fixed-point equation \eqref{eq:affine} associated with the affine recursion $X_n=A_n X_{n-1}+B_n$  has been extensively studied in the literature. We summarize key results on the existence and uniqueness of the solution below, \cite{BurBook16}.
\begin{thm}
    Assume that $\P(A\geq 0, B\geq 0)=1$. If 
    \[
    \E[\log A]<0\quad\mbox{ and }\quad\E[\max\{\log B, 0\}]<\infty,
    \]
    then there exists a unique solution $X$ to 
        \[
    X\stackrel{d}{=}AX+B,\qquad (A,B)\mbox{ and }X\mbox{ are independent}.
    \]
Moreover, $X$ is given by the a.s. convergent series representation
    \[
    X \stackrel{d}{=}\sum_{n=1}^\infty B_n \prod_{k=1}^{n-1} A_k,
    \]
    where $(A_n,B_n)_n$ are independent copies of $(A,B)$. 

Additionally, if $X_n=A_n X_{n-1}+B_n$ for $n=1,2,\ldots$, where $X_0$ is independent of $(A_n,B_n)_n$, then $X_n$ converges in distribution to $X$. 
\end{thm}

We say that a real random variable $Y$ is stochastically majorized by $Z$ and we write $Y\leq_{st} Z$ if $\P(Y\leq t)\geq \P(Z\leq t)$ for all $t\in\R$.  

\begin{lemma}
Let $(A_n, B_n)_{n \geq 1}$ be a sequence of independent copies of a generic pair $(A,B)$.  Let $X_0$ be independent of $(A_n, B_n)_{n \geq 1}$ and define  $X_n = A_nX_{n-1}+B_n$. Assume that $A \geq 0$ a.s. 
\begin{enumerate}
    \item[(i)] If $X_1 \geq_{st} X_0$, then $X_n\geq_{st} X_{n-1}$ for all $n\geq 1$. 
    \item[(ii)] If $X_1 \leq_{st} X_0$, then $X_n\leq_{st} X_{n-1}$ for all $n\geq 1$. 
\end{enumerate}
\end{lemma}
\begin{proof}
(i) We proceed by induction, assume that $X_{n}\geq_{st}X_{n-1}$ for $n\geq 1$. Then, since $A_{n+1}\geq 0$ a.s. and $X_n$ is independent of $(A_{n+1},B_{n+1})$, we obtain
\[
X_{n+1} \stackrel{d}{=}A_{n+1}X_n+B_{n+1}\geq_{st} A_{n+1}X_{n-1}+B_{n+1}\stackrel{d}{=}X_n. 
\]
Point (ii) is proved in the same way. 
\end{proof}
\begin{corol}\label{cor:monotone}
If $X_n$ converges in distribution to $X$, then $X_1\geq_{st} X_0$ implies that $X\geq_{st}X_n$ for all $n\in\N$ and if $X_1\leq_{st}X_0$, then $X\leq_{st} X_n$ for all $n\in\N$. 
\end{corol}

\begin{lemma} \label{ess_supX}
    If $A\geq 0$ a.s., then 
    \[
   \mathrm{ess}\sup(X)= \mathrm{ess}\sup\left( \frac{B}{1-A}\mid A<1\right).
    \]
\end{lemma}
\begin{proof}
Denote $x_+ = \mathrm{ess}\sup(X)$ and 
\[
x_0 = \mathrm{ess}\sup\left( \frac{B}{1-A}\mid A<1\right).
\]
Let $g_{a,b}(t) = a\,t+b$. With 
\[
G(A,B) = \left\{
g_{(a_1,b_1)}\circ\ldots\circ g_{(a_n,b_n)}\colon  (a_i,b_i)\in\mathrm{supp}(A,B), i=1,\ldots,n, n\geq 1
\right\}
\]
by \cite[Proposition 2.5.3]{BurBook16}, we have
\[
\mathrm{supp}(X) = \overline{\left\{ \frac{b}{1-a}\colon g_{a,b}\in G(A,B), a<1 \right\}}.
\]
Thus,
\[
\overline{\left\{
\frac{b}{1-a}\colon (a,b)\in\mathrm{supp}(A,B), a<1
\right\}} \subset \mathrm{supp}(X)
\]
and therefore $x_+\geq x_0$. If $x_0=\infty$, then $x_+=\infty$ and there is nothing to prove. 

Assume that $x_0<\infty$. 
By \cite[Lemma 2.5.1]{BurBook16}, for every $(a, b)\in\mathrm{supp}(A,B)$ we have
\[
a\,\mathrm{supp}(X)+b\subset \mathrm{supp}(X)
\]
and therefore $a x_++b\leq x_+$ for all $(a,b)\in\mathrm{supp}(A,B)$. Thus, $A x_++B\leq x_+$ a.s. and
\[
    B \1_{A\geq 1}\leq x_+ (1-A) \1_{A\geq 1}\leq x_0 (1-A) \1_{A\geq 1}\quad \mbox{a.s.},
\]
where the latter inequality follows from the fact that $x_+\geq x_0$. 

By definition of $x_0$ we have 
\begin{align*}
    B \1_{A<1} &\leq x_0 (1-A) \1_{A<1},\quad \mbox{a.s.}
\end{align*}
and therefore $B \leq x_0 (1-A)$ a.s. 
Thus,
\begin{align*}
    X&\stackrel{d}{=}\sum_{k=1}^\infty A_1\ldots A_{k-1}B_k 
    \leq x_0 \sum_{k=1}^\infty A_1\ldots A_{k-1} (1-A_k) = x_0,\quad\mbox{a.s.},
\end{align*}
which implies that $\P(X\leq x_0)=1$, i.e., $x_0\geq x_+$. 
\end{proof}

We have established above that the right endpoint of the support of $X$ coincides with the right endpoint of the conditional distribution of $B/(1-A)$ given $A<1$. Note that if this endpoint is finite, the asymptotic behavior of $t\mapsto\P(X>t)$ as $t\to\infty$ becomes trivial, and this case should thus be excluded from further analysis. It turns out that suitable conditions involving $g(1)$ or $\lambda^\ast$ are sufficient to ensure this exclusion, but only under the assumption that $\P(A\in[0,1))=1$.  In what follows, we explain that restricting our analysis to such distributions does not lead to any significant loss of generality.

 By \cite[Theorem 4.1]{GG96}, if $\P(A>1)>0$, then the tail of $X$ is at least of power‐law type:
\[
\liminf_{t\to\infty}\frac{\log\P(X>t)}{\log t}>-\infty.
\]
Since this behavior lies outside the light-tail regime we consider, we henceforth restrict to the case where $A\in[0,1]$ a.s. Moreover, when $\P(A=1)>0$, and under suitable nondegeneracy conditions, \cite[Theorem 1.7]{AIR09} shows that the moment generating function $\E[\exp(tX)]$ is finite if and only if
\[
\E[e^{t B}\1_{A=1}]<1.
\]
In addition, \cite[Lemma 5]{DG06} establishes that
\[
\liminf_{t\to\infty} \frac{-\log\P(X>t)}{t} = \sup\{t\in\R\colon \E[e^{t X}]<\infty\} = \sup\{t\in\R\colon \E[e^{t B}\1_{A=1}]<1\}=:t_0.
\]
If $t_0<\infty$, the analysis for the case $\P(A=1)>0$ is complete. Under $\P(A=1)>0$, we have $t_0=\infty$ if and only if $\P(B=0\mid A=1)=1$.  In this degenerate scenario, one can show that
\[
X \stackrel{d}{=}\tilde{A} X+\tilde{B}, \qquad X\mbox{ and }(\tilde{A},\tilde{B})\mbox{ are independent},
\]
where $(\tilde{A},\tilde{B}) \stackrel{d}{=}(A,B)\mid \{A<1\}$. Indeed, for a bounded continuous function $f$ we have
\begin{align*}
\E[f(X)] &= \E[f(AX+B)] = \E[f(AX+B)\1_{A=1}] + \E[f(AX+B)\1_{A<1}] \\
& = \E[f(X)]\,\P(A=1) + \E[f(\tilde{A}X+\tilde{B})]\,\P(A<1),
\end{align*}
which immediately implies that $\E[f(X)]=E[f(\tilde{A}X+\tilde{B})]$. 

Therefore, to avoid these solved cases, we eventually assume that $\P(A\in[0,1))=1$.

\begin{lemma}\label{lem:g(1)}
Let $g$ be a local dependence measure of $(A,B)$ with $\P(A\in[0,1),B\geq0)=1$ for $f\in\mathcal{R}^\rho$, $\rho>0$. Then:
\begin{enumerate}
    \item[(i)] If $g(1)<\infty$, then $\mathrm{ess}\sup\left( \frac{B}{1-A}\mid A<1\right)=\infty$.
    \item[(ii)] If $\lambda^\ast<\infty$, then $g(1)<\infty$. 
\end{enumerate}
\end{lemma}
\begin{proof}
\begin{enumerate}
    \item[(i)]First we notice that since $A<1$ a.s., then $\mathrm{ess}\sup\left( \frac{B}{1-A}\mid A<1\right)=\mathrm{ess}\sup\left( \frac{B}{1-A}\right)$. Moreover,
\[
\lim_{t\to\infty} \frac{-\log\P\left(\frac{B}{1-A}>t\right)}{f(t)} =  \lim_{t\to\infty} \frac{-\log\P(At+B>t)}{f(t)} = g(1)<\infty.
\]
Then for sufficiently large $t$ we get that $\P(\frac{B}{1-A}>t)>0$, which proves the assertion.
    \item[(ii)] Condition $\lambda^\ast<\infty$ implies that there exists $y_0\in(0,1)$ such that $g(y_0)<\infty$. Since $g$ is nonincreasing, we have $g(1)\leq g(y_0)<\infty$. 
\end{enumerate}
\end{proof}

We will also need the following result, which was proved in \cite[Lemma 6.10]{BKT22}.
\begin{lemma}\label{6.10.}
Suppose that $X \overset{d}{=} AX+B$, where $X$ and $(A,B)$ are independent. For any bounded, uniformly continuous function $f$ on $\R$ and any increasing positive integer sequence $(n_k)_k$, a.s.,
\begin{align*}
    \limsup_{m \to\infty}\frac{1}{m}\sum_{k=1}^m f(X_{n_k}) \geq \E[f(X)].
\end{align*}
\end{lemma}

\section{\texorpdfstring{Tails of $X\stackrel{d}{=}AX+B$}{Tails of X=AX+B}}\label{sec:tailmain}

Throughout this section, we assume that $\P(A\in[0,1))=1$ (and thus $\E[\log A]<0$) and $\E[\max\{\log B,0\}]<\infty$, which implies that there exists a unique solution $X$ to 
\[
X \overset{d}{=} AX+B,\qquad\mbox{$X$ and $(A,B)$  are independent}.
\]
Recall that our standing assumption is that both $A$ and $B$ are nonzero. The main result in this section is as follows.

 \begin{thm}\label{theorem4.1}
Let $g$ be an admissible $\LDM_f^\rho$ for the nonnegative random variables $(A,B)$, where $f \in\ \mathcal{R}^{\rho}$ and $\rho > 0$. Then $X$ is $\ED_{f}(\lambda^\ast)$, where $\lambda^\ast\in[0,\infty]$ is defined in \eqref{eq:lambdastar}.
\end{thm}

The proof of Theorem \ref{theorem4.1} will rely on several lemmas.
\begin{lemma}\label{lem:Zk}
If there exists $\kappa>0$ such that $\kappa<\phi_\rho(\kappa)$, then there exists  a nonnegative $\ED_f(\kappa)$-random variable $Z_\kappa$ such that 
\begin{align}\label{**}
Z_\kappa \geq_{st} AZ_\kappa+B,\qquad Z_\kappa\mbox{ and }(A,B)\mbox{ are independent}.
\end{align}
\end{lemma}

\begin{proof}
Let $Z_0$ be a nonnegative random variable independent of $(A, B)$, with the distribution defined by (recall that, without loss of generality, we assumed that $f$ is strictly increasing and continuous)
\begin{align*}
\P(Z_0 > t) = \exp(-\kappa f(t)),\quad t>0.
\end{align*}
Clearly, $Z_0$ is $\ED_f(\kappa)$.

Thus, by Theorem \ref{nowe3.9}, we have
\begin{align*}
\lim_{t \to\infty} \frac{-\log\P(AZ_0+B>t)}{f(t)} = \phi_{\rho}(\kappa) > \kappa = \lim_{t \to\infty}\frac{-\log \P(Z_0>t)}{f(t)}, 
\end{align*}
which implies that 
there exists $M>0$ such that 
\begin{align}\label{(2.2)}
\forall\, t\geq M \,\,\, \P(Z_0>t) \geq \P(AZ_0+B>t). 
\end{align}
Let $Z_\kappa$ be a random variable, independent of $(A, B)$, with the distribution defined by 
\begin{align*}
\P(Z_\kappa \in \cdot) = \P(Z_0 \in \cdot \mid Z_0 > M). 
\end{align*}
For $t\geq M$, we obtain
\begin{align*}
\P(AZ_\kappa+B > t) & = \P(AZ_0+B >t\mid Z_0 > M)  \leq \frac{\P(AZ_0+B>t)}{\P(Z_0>M)} \\& 
\stackrel{\eqref{(2.2)}}{\leq} \frac{\P(Z_0>t)}{\P(Z_0>M)} = \P(Z_0 > t \mid Z_0>M)  =\P(Z_\kappa>t),
\end{align*}
while for $t< M$, we have $\P(Z_\kappa > t)= 1\geq \P(A Z_\kappa+B>t)$. Thus, 
\eqref{**} holds true.

For $t\geq M$, we have as $t\to\infty,$
\begin{align*}
\frac{-\log \P(Z_\kappa>t)}{f(t)}  = \frac{-\log \P(Z_0 > t \mid Z_0>M)}{f(t)}  = \frac{-\log \P(Z_0>t)+\log \P(Z_0>M)}{f(t)} \to \kappa.
\end{align*}
\end{proof}

\begin{lemma}\label{lemma4.3}
Under the assumptions of Theorem \ref{theorem4.1}, 
\[
\liminf_{t \to\infty} \left(-f(t)^{-1} \log \P(X>t) \right)\geq \lambda^\ast.
\]
\end{lemma}
\begin{proof}
The inequality holds trivially in the case $\lambda^\ast  = 0$ and in the case where $X$ has a bounded support (so that $\ess\sup\left(\frac{B}{1-A}\mid A<1\right)<\infty$, recall Lemma \ref{ess_supX}), in which we have $\lambda^\ast=\infty$. 

Assume that $\lambda^\ast \in(0,\infty]$ and $\ess\sup\left(\frac{B}{1-A}\mid A<1\right)=\infty$, which, by Lemma \ref{ess_supX}, implies that $\P(X>t)>0$ for any $t\in\R$. 
Moreover, since $\lambda^\ast >0$, we must have $g(0)>0$. The admissibility of $g$ ensures that there exists $\kappa > 0$ such that $\kappa < \phi_\rho(\kappa)$. 
Let $Z_\kappa$ be a random variable whose distribution is constructed in Lemma \ref{lem:Zk}. Let $X_0 = Z_\kappa$. From \eqref{**}, we obtain  $X_1 \leq_{st} X_0$, 
which, by Corollary \ref{cor:monotone},  implies  
\begin{align*}
\forall\, t \in \R \quad \frac{-\log \P(X>t)}{f(t)} \geq \frac{-\log \P(X_0>t)}{f(t)}. 
\end{align*}
Taking $\liminf_{t \to\infty}$ on both sides, we get 
\begin{align*}
\liminf_{t \to\infty} \frac{-\log \P(X>t)}{f(t)} \geq \liminf_{t \to\infty}\frac{-\log \P(X_0>t)}{f(t)} =  \kappa. 
\end{align*}
Therefore, 
\begin{align*}
\liminf_{t \to\infty}\frac{-\log \P(X>t)}{f(t)} \geq \sup \left\{\kappa >0\colon  \phi_\rho(\kappa) > \kappa \right\}.
\end{align*}
By Lemma \ref{nowe3.15}, we have 
\[
\left\{c \geq 0\colon \phi_\rho(c) \geq c \right\} = \begin{cases}
    [0,\lambda^\ast],& \lambda^\ast<\infty,\\
    [0,\infty),& \lambda^\ast=\infty. 
\end{cases}
\]
However, by Lemma \ref{lem:phi-prop} (v), the set $\left\{c\geq 0\colon  \phi_\rho(c) = c \right\}$ contains at most two elements. Therefore, 
\[
\sup\left\{\kappa > 0\colon  \phi_\rho(\kappa) > \kappa \right\} = \sup\left\{c \geq 0\colon \phi_\rho(c) \geq c \right\} =\lambda^\ast. 
\]
\end{proof}

\begin{lemma}\label{lemma4.5}
Under the assumptions of Theorem \ref{theorem4.1}, if additionally 
\[
s = \limsup_{t \to\infty} \left(-f(t)^{-1}\log \P(X>t)\right) < \infty,
\]
then $s \leq \lambda^\ast $.
\end{lemma}
\begin{proof}
For all $t>0$ and $y>0$, we have
\begin{align}\label{(5.2)}
\begin{split}
\P(X>t) &= \P(AX+B>t) \geq \P(AX+B>t, X > ty) \\
& \geq \P(Aty+B>t, X>ty) = \P(Aty+B>t)\P(X>ty). 
\end{split}
\end{align}
Thus, for $y> 0$, 
\begin{align*}
s & 
\leq \limsup_{t \to\infty\ }\frac{-\log \P(Aty+B>t) -\log \P(X>ty)}{f(t)} \\ 
&\leq \limsup_{t \to\infty\ }\frac{-\log \P(Aty+B>t)}{f(t)} + \limsup_{t \to\infty }\frac{-\log \P(X>ty)}{f(ty)}\frac{f(ty)}{f(t)} =
g(y) + s\,y^{\rho}. 
\end{align*}
Now, using the assumption that $s < \infty$, we obtain $s(1-y^{\rho}) \leq g(y)$. 
Therefore
\begin{align*}
s \leq \inf_{y\in(0,1)}\left\{\frac{g(y)}{1-y^{\rho}}\right\} = \lambda^\ast.
\end{align*}
\end{proof}

\begin{lemma}\label{lem:s<oo}
Under the assumptions of Theorem \ref{theorem4.1}, if additionally $\lambda^\ast  < \infty$, then $\limsup_{t \to\infty}\left(-f(t)^{-1}\log \P(X>t)\right) < \infty$. 
\end{lemma}
\begin{proof}
The assumption
$\lambda^\ast < \infty$ implies that there exists $y\in(0,1)$ such that $g(y)< \infty$.  Fix $\eta>0$. By the definition of $g$, we conclude that 
\begin{align*}
\exists\, M > 0\,\,\, \forall\, t \geq M \quad \frac{-\log \P(Aty+B>t)}{f(t)} \leq g(y)+\eta. 
\end{align*}
Using \eqref{(5.2)}, we obtain for $t>0$ and $k=0,1,\ldots$,
\begin{align*}
-\log \P(X>ty^k) +  \log \P(X>ty^{k+1}) & \leq -\log \P(Aty^{k+1}+B>ty^k) \\ 
& \leq f(ty^k)(g(y)+\eta),
\end{align*}
provided $ty^k \geq M$. 
Summing these inequalities for $k=0,\ldots,n \in\mathbb{N}$, we get
\begin{align}\label{(7.2)}
-\log \P(X>t)+\log \P(X>ty^{n+1}) \leq (g(y)+\eta)\sum_{k=0}^{n}f(ty^k), 
\end{align}
provided $ty^k \geq M$ for $k=0,\ldots,n$. 
Set $n=n_t = \left\lfloor \log(\frac{M}{t}) / \log (y) \right\rfloor$. It is easy to verify that
\begin{align}\label{(8.2)}
ty^{n_t} \geq M \geq ty^{n_t+1}.
\end{align}
Thus, \eqref{(7.2)} implies that 
\begin{align*}
\limsup_{t \to\infty}\left\{\frac{-\log \P(X>t)}{f(t)} + \frac{\log \P(X>ty^{n_t+1})}{f(t)}\right\} \leq (g(y)+\eta)\limsup_{t \to\infty}\sum_{k=0}^{n_t}\frac{f(ty^k)}{f(t)}. 
\end{align*}
Using Potter bounds, Lemma \ref{granica} (iv), for $C=2$ and $\delta =\rho/2$, we conclude that for sufficiently large $t$,
\begin{align*}
\frac{f(ty^k)}{f(t)} \leq 2 y^{k\rho/2}.
\end{align*}
Since $y\in(0,1)$, the series $\sum_{k=0}^{\infty}\frac{f(ty^k)}{f(t)}$ is finite as $t\to\infty$. Finally, we observe that
\begin{multline*}
\limsup_{t \to\infty}\frac{-\log \P(X>t)}{f(t)} + \liminf_{t \to\infty}\frac{\log \P(X>ty^{n_t+1})}{f(t)}
\\ \leq \limsup_{t \to\infty}\left\{ \frac{-\log \P(X>t)}{f(t)} + \frac{\log \P(X>ty^{n_t+1})}{f(t)}\right\} < \infty.
\end{multline*}
By \eqref{(8.2)}, we obtain $\P(X>ty^{n_t+1}) \geq \P(X>M)$.
Since $\lambda^\ast < \infty$, by Lemma \ref{lem:g(1)}, we have $\mathrm{ess}\sup\left( \frac{B}{1-A}\mid A<1\right)=\infty$, which, by Lemma \ref{ess_supX}, implies that $\mathrm{ess}\sup(X) = \infty$.
Therefore $\P(X>M)>0$ and  the second term on the left-hand side above is $0$. This concludes the proof.
\end{proof}

\begin{proof}[Proof of Theorem~\ref{theorem4.1}]
If $\lambda^\ast=\infty$, then $X$ is $\ED_f(\lambda^\ast)$ by Lemma \ref{lemma4.3}.  If $\lambda^\ast<\infty$, then Lemma \ref{lemma4.3} gives the lower bound, which, by Lemma \ref{lem:s<oo}, is the same as the upper bound of Lemma \ref{lemma4.5}. 
\end{proof}

\section{\texorpdfstring{Upper envelope for $(X_n)_n$}{Upper envelope}}\label{sec:UpperEnvelope}

Similarly as in the previous section, we assume that $(A,B)$ are nonzero nonnegative random variables such that $\P(A\in[0,1))=1$ and $\E[\max\{\log B,0\}]<\infty$ and that $X$ is the unique solution $X$ to \eqref{eq:affine}. 
Let $(A_n, B_n)_{n \geq 1}$ be a sequence of independent copies of $(A,B)$. We set  $X_0 = 0$ and consider the sequence $X_n = A_nX_{n-1}+B_n$ for $n\geq1$. 

In view of Lemma \ref{granica} (v), without loss of generality, we assume that $f$ is continuous and strictly increasing so that $f^{-1}$ is well defined.
\begin{thm}\label{nowe6.1}
Let $g$ be an admissible $\LDM_f^\rho$ for $(A,B)$, where $f \in\ \mathcal{R}^{\rho}$ and $\rho > 0$.  Assume that $\lambda^\ast \in (0, \infty)$. Then, almost surely 
\begin{align*}
\limsup_{n\to\infty} \frac{X_n}{f^{-1}(\log n)} = (\lambda^\ast )^{-1/\rho}.
\end{align*}
\end{thm}

All lemmas in this section implicitly make the same assumptions as those in Theorem \ref{nowe6.1}.

\begin{lemma}\label{nowe6.2}
Almost surely, we have
\begin{align*}
\limsup_{n \to\infty} \frac{X_n}{f^{-1}(\log n)} \leq (\lambda^\ast)^{-1/\rho}.
\end{align*}
\end{lemma}
\begin{proof}
Fix $\epsilon > 0$. There exists $\delta \in (0,1)$ such that $\gamma := (1-\delta)(1+\epsilon)>1$. 
By Theorem \ref{theorem4.1}, $X$ is $\ED_{f}(\lambda^\ast)$. Since $\lambda^\ast<\infty$, there exists $M>0$ such that 
\begin{align*}
\forall\,t \geq M \,\,\, \P(X>t) \leq e^{-\lambda^* f(t)(1-\delta)}
\end{align*}
Define 
\[
t_n = f^{-1}\left(\frac{(1+\epsilon)\log n }{\lambda^\ast}\right),\quad n\in\mathbb{N}.
\]
Since $B_1=X_1 \geq_{st} X_0=0$, Corollary \ref{cor:monotone} implies that for all $n \in\N$, it holds that $X \geq_{st} X_n$.
As $t_n\to\infty$, for sufficiently large $n$, we obtain
\begin{align*}
    \P\left(X_n > t_n\right) \leq \P(X>t_n) \leq 
    e^{-(1+\epsilon)(1-\delta)\log n} = n^{-\gamma}.
\end{align*}
Since $\gamma > 1$, we conclude that 
\begin{align*}
    \sum_{n=1}^{\infty} \P\left(X_n > t_n\right) < \infty.
\end{align*}
By the Borel-Cantelli Lemma, we get that almost surely,
    \begin{align*}
\limsup_{n \to\infty} \frac{X_n}{f^{-1}\left(\frac{(1+\epsilon)\log n}{\lambda^\ast }\right)} \leq 1.
\end{align*}
By Lemma \ref{granica} (iv), $f^{-1}$ is regularly varying with index $\frac{1}{\rho}$, and therefore,
\begin{align*}
f^{-1}\left(\frac{(1+\epsilon)\log n }{\lambda^\ast }\right) \sim \left(\frac{1+\epsilon}{\lambda^\ast }\right)^{\frac{1}{\rho}} f^{-1}(\log n).
\end{align*}
Hence, almost surely,
\begin{align*}
    \limsup_{n \to\infty} \frac{X_n}{ f^{-1}(\log n)} \leq \left(\frac{1+\epsilon}{\lambda^\ast }\right)^{\frac{1}{\rho}}.
\end{align*}
By letting $\epsilon \to 0^+$, we obtain the assertion.
\end{proof}

In the following lemma, we assume that $X_0$ is arbitrary but independent of $(A_n,B_n)_{n}$. 
\begin{lemma}\label{nowe6.7}
For any $\delta >0$, there exist $y_\ast \in (0,1)$ and $\tilde{M} > 0$ such that if $ty_\ast^{n-1} \geq \tilde{M}$, then, a.s.,
\begin{align*}
  P(X_n>t \mid X_0) \geq \1_{(ty_*^{n}, \infty)}(X_0) \exp(-(1+\delta)\lambda^\ast f(t)).
\end{align*}

\end{lemma}
Before proving the above result, we first present a simple lemma. 

\begin{lemma}\label{nowe6.3}
For all $n \geq 1$, $y>0$ and $t>0$ we have, a.s.,
\begin{align*}
\P(X_n>t \mid X_0) \geq \1_{(y^n t, \infty)}(X_0) \prod_{k=0}^{n-1}\P(t\,y^kAy+B>t\,y^k).
\end{align*}
\end{lemma}

\begin{proof}
Notice that 
\begin{align*}
    \P(X_n>t \mid X_0)
    & \geq \P(A_nX_{n-1}+B_n>t, X_{n-1} > t y \mid X_0) \\
    & \geq \P(t\, A_n y + B_n > t, X_{n-1} > t y \mid X_0) \\
    &=     \P(t\,Ay+B>t)\P(X_{n-1}>t y \mid X_0).
\end{align*}
We obtain the assertion by iterating the above inequality. 
\end{proof}

Now we are ready to present the proof of Lemma \ref{nowe6.7}.
\begin{proof}[Proof of Lemma \ref{nowe6.7}]
    Fix $\alpha>0$ and let $y_* \in (0,1)$ be such that
    \begin{align}\label{8.9.1}
        \lambda^\ast  \leq \frac{g(y_*)}{1-y_*^\rho} \leq \lambda^\ast (1+\alpha).
    \end{align}
    Such $y_*$ exists as a consequence of the definition of $\lambda^\ast = \inf_{y\in(0,1)}\left\{g(y)/(1-y^{\rho})\right\}$ and the fact that $\lambda^\ast>0$. Clearly $g(y^\ast)>0$. 
    By Lemma \ref{nowe6.3}, we have, a.s., 
    \begin{align*}
        \P(X_n>t \mid X_0) \geq \1_{(y_*^n t, \infty)}(X_0) \prod_{k=0}^{n-1}\P\left(ty_*^kAy+B>ty_*^k\right). 
    \end{align*}
    By the definition of $g$, there exists $M>0$ such that 
    \begin{align*}
    \forall \,t \geq M \,\,\, \P(Aty_* + B>t) \geq \exp(-(1+\alpha)g(y_*)f(t)).
    \end{align*}
    Thus, if $t y_\ast^k\geq M$ for $k=0,\ldots,n-1$,  (which is equivalent to $ty_*^{n-1} \geq M$), we obtain 
    \begin{align}\label{8.9.5}
    \P(X_n>t \mid X_0) \geq \1_{(y_*^n t, \infty)}(X_0) \exp\left(-(1+\alpha)g(y_*)\sum_{k=0}^{n-1}f(ty_*^k)\right).
    \end{align}
    Fix $\eta$ in $(0, \rho)$. 
    By Lemma \ref{granica} (iii), there exists $M_1 > 0$ such that for all $z \in (0,1]$ and for all $t>0$ such that $tz \geq M_1$, it holds that 
    \begin{align*}
        \frac{f(tz)}{f(t)} \leq (1+\alpha)z^{\rho - \eta}
    \end{align*}
    Assume that $ty_{\ast}^{n-1} \geq M_1$. Then for $k = 0, ..., n-1$, we have $ty_{\ast}^k \geq ty_{\ast}^{n-1} \geq M_1$. Therefore, 
    \begin{align}\label{eq:M1}
         f(ty_{\ast}^k) \leq (1+\alpha)y_{\ast}^{k(\rho - \eta)} f(t),\qquad k = 0, ..., n-1.
\end{align}
    Notice that for all $n \in \N$ and sufficiently small $\eta$, we have
    \begin{align}\label{eq:eta}
       \frac{1-y_*^{n(\rho - \eta)}}{1-y_*^{\rho - \eta}} \leq \frac{1}{1-y_*^{\rho - \eta}} \leq \frac{1+\alpha}{1-y_\ast^\rho},
    \end{align}
    where the first inequality above follows from the fact that $y_\ast\in(0,1)$ and the latter is true for sufficiently small $\eta$. 
    Take $\eta\in(0,\rho)$ satisfying the condition above.
    
    Define $\tilde{M} = \max \left\{M_1,M\right\}$ and 
    assume $ty_*^{n-1} \geq \tilde{M}$. 
    Then, by \eqref{8.9.5}, \eqref{eq:M1}, \eqref{eq:eta} and \eqref{8.9.1}, we obtain 
    \begin{align*}
        \P(X_n>t \mid X_0) &\stackrel{\eqref{eq:M1}}{\geq} \1_{(y_*^n t, \infty)}(X_0) \exp\left(-(1+\alpha)^2 g(y_*)\sum_{k=0}^{n-1}y_*^{k(\rho - \eta)}f(t)\right) \\
        &= 
        \1_{(y_*^n t, \infty)}(X_0) \exp\left(-(1+\alpha)^2g(y_*) \frac{1-y_*^{n(\rho - \eta)}}{1-y_*^{\rho - \eta}}f(t)\right) \\ 
        & \stackrel{\eqref{eq:eta}}{\geq}
        \1_{(y_*^n t, \infty)}(X_0) \exp\left(-(1+\alpha)^3g(y_*)\frac{1}{1-y_*^\rho}f(t)\right) \\
        & \stackrel{\eqref{8.9.1}}{\geq} \1_{(y_*^n t, \infty)}(X_0) \exp\left(-(1+\alpha)^4\lambda^\ast f(t)\right).
    \end{align*}
    Since, for all $\delta > 0$, it holds that $1+\delta$ is of the form $(1+\alpha)^4$ for some $\alpha > 0$, this concludes the proof.
    \end{proof}

\begin{lemma}\label{noweA.1}
Let $\tilde{M}, \lambda^\ast , \epsilon > 0$ and $y_* \in (0,1)$. Suppose that $f$ is regularly varying with index $\rho > 0$. There exists a strictly increasing sequence $(k_n)_{n=1}^\infty$ of positive integers and a positive constant $c>\tilde{M}y_*$ such that, for $n \geq 1$,
\begin{align}\label{eq:kn}
c \geq f^{-1}\left(\frac{\log k_{n+1}}{\lambda^\ast (1+\epsilon)}\right)y_\ast^{k_{n+1}-k_n} &> \tilde{M} y_\ast.
\end{align}
Moreover, for any $\gamma \in (0,1)$, there exists $K>0$ such that for all $n \geq 1$,
\begin{align}\label{eq:knbound}
    k_n^\gamma \leq Kn.
\end{align}
\end{lemma}

\begin{proof}
Denote $H(t) = f(1/t)$. By definition, $f$ is regularly varying (at $\infty$) with index $\rho>0$ if and only if $H$ is regularly varying at $0$ with index $-\rho$, \cite[page 8]{BGT89}.
Since $f^{-1}(t)=1/H^{-1}(t)$, \eqref{eq:kn} can be rewritten as 
\begin{align*}
H^{-1}\left(\frac{\log k_{n+1}}{\lambda^\ast (1+\epsilon)}\right)\left(\frac{1}{y_\ast}\right)^{k_{n+1}-k_n-1} < \frac{1}{\tilde{M}}
\quad\mbox{and}\quad 
H^{-1}\left(\frac{\log k_{n+1}}{\lambda^\ast (1+\epsilon)}\right)\left(\frac{1}{y_*}\right)^{k_{n+1}-k_n} \geq \frac{1}{c},
\end{align*}
where $c^{-1}<(\tilde{M}y_*)^{-1}$. The existence of a sequence $(k_n)_{n=1}^\infty$ satisfying the two above conditions and \eqref{eq:knbound} was established in \cite[Lemma A.1]{BKT22}.  
\end{proof}

We will use the following version of the Borel–Cantelli lemma, which can be found in \cite[Theorem 5.1.2]{BC12}.
\begin{lemma}\label{lem:borel} 
    Suppose that $(\mathcal{F}_n)_{n\geq 0}$ is a filtration such that $\mathcal{F}_0=\{\emptyset,\Omega\}$. Assume that $A_n\in\mathcal{F}_n$ for all $n\geq 0$. Then,
    \[
    \limsup_{n\to\infty}A_n = \left\{ \sum_{n=1}^\infty \P(A_n\mid \mathcal{F}_{n-1}) =\infty\right\}. 
    \]
\end{lemma}

\begin{lemma}[Kronecker's Lemma]\label{lem:kron}
Assume that $a_n \uparrow \infty$. If $\sum_{n=1}^\infty x_n/a_n$ converges, then $\lim_{n \to\infty} a_n^{-1}\sum_{m=1}^{n}x_m =0$. \\Equivalently, if $\lim_{n \to\infty} a_n^{-1}\sum_{m=1}^{n}x_m \neq 0$, then  $\sum_{n=1}^\infty x_n/a_n$ does not converge.
\end{lemma}

\begin{lemma}\label{nowe6.11}
 Almost surely, we have
\begin{align*}
\limsup_{n\to\infty} \frac{X_n}{f^{-1}(\log n)} \geq (\lambda^\ast)^{-1/\rho}.
\end{align*}
\end{lemma}

\begin{proof}
Fix $\epsilon>0$ and define 
\[
t_n = f^{-1}\left(\frac{\log n}{\lambda^\ast (1+\epsilon)}\right). 
\]
Let $(\mathcal{F}_n)_{n\geq 0}$ denote the natural filtration of the sequence $(X_n)_{n\geq 0}$. By Lemma \ref{lem:borel}, it suffices to show that for a strictly increasing sequence of positive integers $(k_n)_{n\geq 1}$,  we have, a.s.,
\[
\sum_{n=1}^\infty \P(X_{k_n} > t_{k_n}\mid \mathcal{F}_{k_{n-1}}) = \infty, 
\]
which implies that $\P(\limsup_{n\to\infty} A_{k_n} )=1$. Since 
$\limsup_{n\to\infty} A_n \supset \limsup_{n\to\infty} A_{k_n}$, this would prove the assertion.  

 Since $(X_{k_n})_{n\geq 1}$ is a Markov chain, we have
\[
\P(X_{k_{n+1}}>t_{k_{n+1}}\mid \mathcal{F}_{k_{n}}) = \P(X_{k_{n+1}}>t_{k_{n+1}}\mid X_{k_{n}}).   
\]
By Lemma \ref{nowe6.7}, applied to the sequence $(Y_n)_{n\in\mathbb{N}\cup\{0\}}$ defined by $Y_{i}=X_{k_n+i}$, we obtain that for any $\delta>0$, there exist $y_\ast \in (0,1)$ and $\tilde{M}>0$ such that 
    \begin{align} \label{8.11.1}
        \P(X_{k_{n+1}} > t \mid X_{k_n}) \geq \1_{(ty_*^{k_{n+1}-k_n}, \infty)}(X_{k_n})\exp(-(1+\delta)\lambda^\ast f(t)),
    \end{align}
    provided $ty_*^{k_{n+1}-k_n-1} \geq \tilde{M}$.

We use the sequence $(k_n)_{n\geq 1}$ from Lemma \ref{noweA.1}. The condition $t_{k_{n+1}}y_*^{k_{n+1}-k_n-1} \geq \tilde{M}$ is satisfied as a consequence of the lower bound in \eqref{eq:kn}. Therefore, using \eqref{8.11.1} and the upper bound in \eqref{eq:kn}, we obtain, a.s.,
\begin{align*}
\P(X_{k_{n+1}} > t_{k_{n+1}}\mid X_{k_n}) &\geq \1_{(c, \infty)}(X_{k_n})\exp\left(-\tfrac{1+\delta}{1+\epsilon}\log k_{n+1}\right) = \1_{(c, \infty)}(X_{k_n})\frac{1}{k_{n+1}^\gamma},
\end{align*}
where $\gamma := (1+\delta)/(1+\epsilon)$. By decreasing $\delta$ if necessary, we ensure that $\gamma < 1$. Our goal is to show that $\sum_{n\geq 1}  \1_{(c, \infty)}(X_{k_n})k_{n+1}^{-\gamma}$ diverges a.s. By Lemma \ref{lem:kron}, applied to $a_n = k_{n+1}^\gamma$ and $x_n =\1_{(c, \infty)}(X_{k_n})$, to meet this goal, it suffices to show that, a.s.,
\begin{align}\label{eq:goal}
 \limsup_{m \to\infty}\frac{1}{k_{m+1}^\gamma}\sum_{n=1}^m \1_{(c, \infty)}(X_{k_n})>0. 
\end{align}
From \eqref{eq:knbound}, we have $k_n^\gamma \leq Kn$ for some $K>0$ and all $n \geq 1$. Therefore, a.s.,
\begin{align*}% \label{8.11.4}
    \limsup_{m \to\infty}\frac{1}{k_{m+1}^\gamma}\sum_{n=1}^m \1_{(c, \infty)}(X_{k_n}) \geq
    \limsup_{m \to\infty}\frac{K^{-1}}{m+1}\sum_{n=1}^m \1_{(c, \infty)}(X_{k_n}).
\end{align*}
Let $f_c\colon \R \rightarrow \R$ be defined by $f_c(x) =(x-c)/c \1_{[c,2c]}(x)+\1_{(2c,\infty)}(x)$.
Since $f_c$ is bounded and uniformly continuous on $\R$, and since $\1_{(c,\infty)}\geq f_c\geq \1_{(2c,\infty)}$, by Lemma \ref{6.10.} we have
\begin{align*}
    \limsup_{m \to\infty}\frac{1}{m+1}\sum_{n=1}^m \1_{(c, \infty)}(X_{k_n}) & \geq \limsup_{m \to\infty}\frac{1}{m+1}\sum_{n=1}^m f_c(X_{k_n})    \geq \E[f_c(X)] \geq \P(X>2c).
\end{align*}
Since $\P(X>t)>0$ for any $t\in\R$ (recall that under $\lambda^\ast<\infty$, by Lemmas \ref{ess_supX} and \ref{lem:g(1)} , we have $\ess\sup(X)=\infty$), \eqref{eq:goal} follows.
    Therefore, by Lemma \ref{lem:borel}, for every $\epsilon>0$, we have almost surely
    \begin{align*}
\limsup_{n \to\infty} \frac{X_n}{f^{-1}(\frac{\log n}{\lambda^\ast (1+\epsilon)})} \geq 1.
\end{align*}
Proceeding as in the proof of Lemma \ref{nowe6.2}, we obtain
\begin{align*}
    \limsup_{n \to\infty} \frac{X_n}{ f^{-1}(\log n)} \geq \left(\frac{1}{\lambda^\ast (1+\epsilon)}\right)^{\frac{1}{\rho}}.
\end{align*}
The proof is completed by letting $\epsilon \to 0^+$.
\end{proof}

\begin{proof}[Proof of Theorem \ref{nowe6.1}] 
This follows directly from Lemmas \ref{nowe6.2} and \ref{nowe6.11}.
\end{proof}

\section{Example}\label{sec:ex}
In this section, we present a family of distributions of $(A,B)$, which admit an explicit $\LDM$, but is not PQD. Thus, allowing us to illustrate our results on explicit example outside PQD pairs. This family is indexed by a continuous, nondecreasing function $\alpha\colon [0,1)\mapsto [1,\infty)$.

Let $\rho>0$, and consider a random vector $(A,B)$
whose law consists of a unique atom and an absolutely continuous component. Specifically, assume that
$\P(A=0, B=1) = 1 - e^{-1}$,
and that 
\[
\P(A>a, B>b) = \exp\left(-\alpha(a) \beta(b)\right),\qquad a\in[0,1), b\geq 1,
\]
with 
\[
\beta(b)=b^\rho \quad \text{and} \quad \ess\sup(A)=1.
\]
Since $e^{-1}=\P(A>0, B>1)=\exp(-\alpha(0)\beta(1))$, we deduce that $\alpha(0)=1$. In addition, one may verify that $\alpha(1^-)=\infty$, that $\alpha(a) = -\log\P(A>a)$ for $a\geq 0$ and that $\log\P(B>b) = - b^\rho$ for $b\geq 1$. 

Our goal is to compute the $\LDM^\rho_f$ function for $(A,B)$ with $f(t)=-\log\P(B>t)=t^\rho$. 
 Note that $A$ and $B$ are not positively quadrant dependent, so that the results of Section \ref{sec:PQD} do not apply.
 
 By Lemma \ref{3} (i), we have $g(y)=0$ for $y>1$. Fix $y\in[0,1)$ and consider $t>1/(1-y)$. Then,
\begin{align*}
  \P(Aty+B>t) &= \int_0^1 \int_{t(1-y a)}^\infty \frac{\partial^2 \exp(-\alpha(a)\beta(b))}{\partial a\,\partial b}  \,\dd b\,\dd a \\
  & = - \int_0^1 \frac{\partial \exp(-\alpha(a)\beta(b))}{\partial a}\big|_{b=t(1-y a)}  \dd a 
  \\
  & =t^\rho \int_{0}^{1} (1-y a)^\rho e^{-t^\rho (1-y a)^\rho\alpha(a)} \alpha'(a) \dd a \\
&  = t^\rho \int_{1}^{\infty} (1-y \alpha^{-1}(u))^\rho e^{-t^\rho (1-y \alpha^{-1}(u))^\rho u} \dd u,
\end{align*}
where the substitution $u=\alpha(a)$ was used (with $\alpha^{-1}$ denoting the generalized inverse of $\alpha$).  Since $(1-y)^\rho\leq (1-y \alpha^{-1}(u))^\rho \leq 1$, we obtain 
\[
(1-y)^\rho t^\rho \int_{1}^{\infty}  e^{-t^\rho (1-y \alpha^{-1}(u))^\rho u} \dd u \leq  \P(Aty+B>t) \leq t^\rho \int_{1}^{\infty}  e^{-t^\rho (1-y \alpha^{-1}(u))^\rho u} \dd u.
\]
By the Laplace method, it follows that $g$ exists and equals for $y\in(0,1)$,
\[
g(y) = \lim_{t\to\infty} \frac{-\log\P(Aty+B>t)}{t^\rho} = \inf_{u\in[1,\infty)} \{ (1-y \alpha^{-1}(u))^\rho u \} = \inf_{a\in[0,1)}\left\{ (1-y a)^\gamma  \alpha(a) \right\}.
\]
We obtain $g(0)=\alpha(0)=1$ and
\begin{multline*}
\phi_{\rho}(\lambda) = \inf_{y>0} \left\{y^{\rho}\lambda+g(y)\right\}\\
=
\begin{cases}
    \min\{1,\lambda\}, & \rho\in(0,1], \\
    \min\left\{\inf_{a\in[0,1)}\left\{\left(\alpha(a)^{1/(1-\rho)}+a^{\rho/(\rho-1)} \lambda^{1/(1-\rho)}\right)^{1-\rho}\right\},\lambda\right\}, & \rho>1.
\end{cases}
\end{multline*}
Moreover, one can easily show that
\begin{multline*}
\lambda^\ast = \inf_{y\in(0,1)}\left\{\frac{g(y)}{1-y^{\rho}}\right\} = \inf_{a\in[0,1)} \inf_{y\in(0,1)}\left\{\frac{\alpha(a)(1-y a)^\rho }{1-y^{\rho}}\right\} \\
= \begin{cases}
 \inf_{a\in[0,1)}\left\{\alpha(a)\right\}=1, & \rho\in(0,1],\\
 \inf_{a\in[0,1)} \left\{
 \alpha(a) (1-a^{\rho/(\rho-1)})^{\rho-1}
 \right\}, & \rho>1.
\end{cases}
\end{multline*}

Consider the case $\rho>1$. Then, if $X$ is a solution to \eqref{eq:affine}, by Theorem \ref{theorem4.1}, we obtain
\begin{align}\label{eq:our}
\lim_{t\to\infty} \frac{\log\P(X>t)}{t^\rho} = -  \inf_{a\in[0,1)} \left\{
 \alpha(a) \left(1-a^{\rho/(\rho-1)}\right)^{\rho-1}
 \right\}.
\end{align}
We are going to relate this result to the findings of \cite{BK18}. In \cite[Theorem 5.1]{BK18}, it was shown that
\begin{align}\label{eq:prev}
\liminf_{t\to\infty} \frac{\log\P(X>t)}{h(t)} \geq -
\left( s \left(1-\left(1-\frac1{s}\right)^{\gamma/(\gamma-1)}\right)\right)^{\gamma-1},
\end{align}
where 
\begin{align*}
h(t) = \inf_{s\geq 1}\left\{
-s\log\P\left( A>1-\frac1s, B>\frac{t}{s}\right)
\right\},
\end{align*}
$\gamma$ is the index of regular variation of $h$ and $s=\lim_{t\to\infty}\sigma(t)$, where $\sigma$ is any function satisfying 
\[
h(t) = -\sigma(t)\log\P\left( A>1-\frac{1}{\sigma(t)}, B>\frac{t}{\sigma(t)}\right)+o(1),\qquad t\to\infty. 
\]
In our setting, for $t>1$ one can write
\begin{align*}
h(t) =  \inf_{s\in[1,t]}\left\{ s\,\alpha\left(1-\frac1s\right)\beta\left(\frac ts\right)
\right\} = t^\rho  \inf_{s\in[1,t]}\left\{ s^{1-\rho}\, \alpha\left(1-\frac1s\right)
\right\}.
\end{align*}
Now, we consider two cases: 
\begin{itemize}
    \item[(a)] $\alpha(a) = (1-a)^{1-\rho}$ with $\rho>1$, 
    \item[(b)] $\alpha(a) = \exp(a/(1-a))$ with $\rho>2$.
\end{itemize}
In case (a) one immediately obtains $h(t) = t^\rho$ for $t>1$. In this situation, one can take $\sigma(t)=\min\{s,t\}$ for any $s\geq 1$. Consequently, $\gamma=\rho$ and \eqref{eq:prev} becomes (after taking $\sup_{s\geq 1}$ of both sides)
\[
\liminf_{t\to\infty} \frac{\log\P(X>t)}{t^\rho} \geq - \inf_{s\geq 1} \left\{ 
\left[s \left(1-\left(1-\frac1{s}\right)^{\rho/(\rho-1)}\right)\right]^{\rho-1}\right\}.
\]
A direct calculation shows that this lower bound agrees with the expression in \eqref{eq:our}.

Case (b). 
Then,
\[
h(t) = t^\rho \inf_{s\in[1,t]}\left\{ s^{1-\rho} e^{s-1} \right\}.
\]
For $t>\rho-1>1$,  the infimum is attained at $s=\rho-1$, yielding 
\[
h(t) = t^\rho (\rho-1)^{1-\rho}e^{\rho-2}.
\]
Again, one  obtains $\gamma=\rho$, and by taking $\sigma(t)=\rho-1=s$, \eqref{eq:prev} gives
\begin{align*}
\liminf_{t\to\infty} \frac{\log\P(X>t)}{t^\rho} & = \liminf_{t\to\infty} \frac{\log\P(X>t)}{h(t)}\cdot\frac{h(t)}{t^\rho}  \\
&\geq -\left( (\rho-1) \left(1-\left(1-\frac1{\rho-1}\right)^{\rho/(\rho-1)}\right)\right)^{\rho-1}\cdot
(\rho-1)^{1-\rho}e^{\rho-2} \\
&=-e^{a/(1-a)} \left(1-a^{\rho/(\rho-1)}\right)^{\rho-1}\big|_{a=1-(\rho-1)^{-1}}.
\end{align*}
On the other hand, by \eqref{eq:our}, we have
\begin{align*}
\lim_{t\to\infty} \frac{\log \P(X>t)}{t^\rho} &= -\inf_{a\in[0,1)}\left\{ e^{a/(1-a)} \left(1-a^{\rho/(\rho-1)}\right)^{\rho-1}\right\} \\
& >- e^{a/(1-a)} \left(1-a^{\rho/(\rho-1)}\right)^{\rho-1}\big|_{a=1-(\rho-1)^{-1}}.
\end{align*}
Thus, the lower bound provided by Theorem 5.1 in \cite{BK18} is, in general, not optimal.

\bibliographystyle{plain}
\bibliography{Bibl}

\end{document}